\documentclass[12pt]{amsart}
\usepackage{amsmath,amssymb,enumerate}
\newtheorem{theorem}{Theorem}[section]
\newtheorem{claim}{}[theorem]
\newtheorem{lemma}[theorem]{Lemma}

\theoremstyle{definition}

\newcommand{\bZ}{\mathbb Z}

\newcommand{\cB}{\mathcal{B}}
\newcommand{\cC}{\mathcal{C}}
\newcommand{\cD}{\mathcal{D}}

\newcommand{\cP}{\mathcal{P}}

\newcommand{\cQ}{\mathcal{Q}}

\newcommand{\cU}{\mathcal{U}}

\newcommand{\cW}{\mathcal{W}}

\newcommand{\id}{\mathrm{id}}

\newcommand{\bal}{\mathsf{b}}
\newcommand{\ub}{\mathsf{u}}

\newcommand{\oa}[1]{\vec{#1}}
\newcommand{\kub}[1]{K^{\mathbf{u}}_{#1}}
\newcommand{\kosc}[1]{K^{\mathbf{o}}_{#1}}
\newcommand{\ka}[2]{K^{#1}_{#2}}
\newcommand{\bu}{\mathbf{u}}
\newcommand{\bo}{\mathbf{o}}
\allowdisplaybreaks[1]
\author{Peter Nelson}
\author{Sophia Park}
\address{Department of Combinatorics and Optimization,
University of Waterloo, Waterloo, Canada}

\title{A Ramsey Theorem for Biased Graphs}
\begin{document}
\maketitle

\begin{abstract}
	A \emph{biased graph} is a pair $(G,\cB)$, where $G$ is a graph and $\cB$ is a collection of `balanced' circuits of $G$ such that no $\Theta$-subgraph of $G$ contains precisely two balanced circuits. We prove a Ramsey-type theorem, showing that if $(G,\cB)$ is a biased graph which $G$ is a very large complete graph, then $G$ contains a large complete subgraph $H$ such that the set of balanced cycles within $H$ has one of three specific, highly symmetric structures, all of which can be described naturally via group-labellings. 
\end{abstract}

\section{Introduction}

A \emph{$\Theta$-subgraph} of a graph $G$ is a connected subgraph of $G$ whose edge set is the union of three edge-disjoint $xy$-paths for some pair of distinct vertices $x$ and $y$. Such a subgraph contains exactly three circuits of $G$. A \emph{biased graph} is a pair $(G,\cB)$ where $G$ is a graph and $\cB$ is a collection of circuits of $G$, called the `balanced' circuits, so that no $\Theta$-subgraph of $G$ 
contains precisely two circuits in $\cB$.

As was shown in Zaslavsky's work [\ref{zaslavsky}] defining these objects, a natural class of biased graphs arises by assigning orientations and labels from a group  to the edges of a graph $G$, and indeed, biased graphs arose as an abstraction of these `group-labelled graphs'. We explain the construction briefly here. For a group $\Gamma$, a \emph{$\Gamma$-labelled graph} is a pair $(\oa{G},\gamma)$, where $\oa{G}$ is an orientation of an undirected graph $G = (V,E)$, (that is, $G$ together with an assignment of a `head' in $\{u,v\}$ to each undirected edge $uv \in E$, towards which the edge is `oriented') and $\gamma\colon E \to \Gamma$ is a function. 

For each circuit $C$ of $G$ with vertices $v_1,\dotsc,v_k$ and edges $e_1,\dotsc,e_k$ where $v_i,v_{i+1}$ are the ends of $e_i$ and $v_{k+1} = v_1$, let $\pi(C) = \prod_{i=1}^k \gamma'(e_i)$, where $\gamma'(e_i) = \gamma(e_i)$ if $e_i$ is oriented from $v_i$ to $v_{i+1}$, and $\gamma'(e_i) = \gamma(e_i)^{-1}$ otherwise. We say that $C$ is \emph{balanced} if $\pi(C)$ is the identity in $\Gamma$; although $\pi(C)$ may depend on the particular choice of indexing of the vertices in $C$, associativity implies that the definition of balanced does not, even if $\Gamma$ is nonabelian. Zaslavsky showed that $(G,\cB)$ is always a biased graph for the set $\cB$ of circuits of $G$ that are balanced in the above sense. It is known, however, that not all biased graphs arise from group-labelled graphs in this way; in fact, [\ref{dfp}] shows that there exist infinite antichains (with respect to biased subgraph containment) of biased (multi-)graphs on three vertices that do not. 

A \emph{biased subgraph} of $(G,\cB)$ is a pair $(H,\cB')$ where $H$ is a subgraph of $G$, and $\cB'$ is the intersection of $\cB$ with the set of circuits of $H$. We prove the following Ramsey-type theorem for biased graphs; $K_n$ denotes the complete graph with vertex set $[n] = \{1,\dotsc,n\}$. 

\begin{theorem}\label{maingroup}
	For every integer $s \ge 1$, there exists an integer $n$ such that, if $(G,\cB)$ is a biased graph with $G \cong K_n$, then $(G,\cB)$ has a biased subgraph $(H,\cB')$ with $H \cong K_s$ that arises from a $\Gamma$-labelled graph for some finite cyclic group $\Gamma$. 
\end{theorem}

In fact, we refine the conclusion to show that $(H,\cB')$ arises from a graph with one of three very specific types of group-labelling, which we now define. 

Let $\oa{K_n}$ be the orientation of $K_n = ([n],E)$, where every edge with ends $i < j$ is oriented from $i$ to $j$. Let $m = \binom{n}{2}$ and $\Gamma$ be a cyclic group of order at least $2^m$ with generator $g$. Let $S \subseteq \Gamma$ be a set of size $m$ such that no two distinct subsets of $S$ have the same product; such a set $S$ exists, as we can take $S = \{g^{2^i} \colon 0 \le i < m\}$. Let $\gamma^{\mathbf{u}} \colon E \to \Gamma$ be a labelling assigning distinct elements of $S$ to each edge.  Let $\gamma^{\mathbf{o}} \colon E \to \Gamma$ be a labelling such that every edge receives a label in $S$, and two edges receive the same label if and only if they have the same head (that is, larger end). 

For each integer $a \ge 0$, let $\Gamma$ be a multiplicative cyclic group that is infinite if $a = 0$ and has size $a$ otherwise, and let $g$ be a generator for $\Gamma$. Let $\gamma^a \colon E \to \Gamma$ assign a label of $g$ to every edge. 

Let $\kub{n}, \kosc{n},\ka{a}{n}$ be the biased graphs of the group-labelled graphs corresponding to $\oa{K_n}$ and the labellings $\gamma^{\mathbf{u}},\gamma^{\mathbf{o}}$ and $\gamma^a$ (for $a \ge 0$) respectively. While the first two ostensibly depend on the choice of $\Gamma$ and $S$, Lemma~\ref{basicclasses} will show that in fact they do not. We can now state a refinement of our main theorem. 

\begin{theorem}\label{mainsimp}
	For every integer $s \ge 1$, there exists an integer $n$ such that, if $(G,\cB)$ is a biased graph with $G \cong K_n$, then $(G,\cB)$ has a biased subgraph isomorphic to either $\kub{s}$, $\kosc{s}$, or $\ka{a}{s}$ for some $a > 0$.
\end{theorem}

In all three cases, there are simple combinatorial descriptions for the set of balanced cycles in the subgraph. We state these here, and will prove them in Lemma~\ref{basicclasses}. For a $k$-circuit $C$ of $K_n$, let $v_1, \dotsc, v_k$ be the unique ordering of $V(C)$ for which $E(C) = \{v_1v_2,\dotsc,v_{k-1}v_k,v_kv_{1}\}$, and $v_1 = \min(C)$ while $v_2 < v_k$. Set $v_0 = v_k$. For $i \in [k]$, we say that the edge $v_{i-1}v_i$ is \emph{positive} in $C$ if $v_i > v_{i-1}$, and \emph{negative} otherwise. We will show that
\begin{itemize}
	\item $\kub{s}$ has no balanced circuits, 
	\item a circuit $C$ is balanced in $\kosc{s}$ if and only if the edges of $C$ alternate between positive and negative, and 
	\item a circuit $C$ is balanced in $\ka{a}{s}$ if and only if $a|d$, where $d$ is the difference between the number of positive edges in $C$ and the number of negative edges in $C$. 
\end{itemize}

Note that if $a = 1$ in the third case, then every circuit is balanced, and if $a = 2$ then precisely the even circuits are balanced; these natural cases are therefore both possible outcomes of Theorem~\ref{mainsimp}.

The function $n = n(s)$ implicitly defined in Theorem~\ref{mainsimp} depends on Ramsey numbers for $4$-colouring the edges of a complete $4$-uniform hypergraph, and therefore corresponds to a tower of exponents of constant height. One could shorten our proof in exchange for much worse numbers (invoking Ramsey's theorem for $f(s)$-uniform hypergraphs with $g(s)$ colours, where $f$ and $g$ are respectively linear and exponential in $s$). We also prove a bipartite version of our main theorem. Here the outcomes simplify and our bound is small enough to include explicitly. 

\begin{theorem}\label{mainbp}
	Let $t \ge 1$. If  $n \ge 2^{2^{2^{4t}}}$ and $(G,\cB)$ is a biased graph with $G \cong K_{n,n}$, then $(G,\cB)$ has a biased subgraph $(H,\cB')$ with $H \cong K_{t,t}$ such that $\cB'$ is either empty, or the set of all circuits in $H$. 
\end{theorem}
As part of the work we do to get our better bound in Theorem~\ref{mainsimp}, we obtain a result of independent interest, Lemma~\ref{constantlabelling}, that gives a nearly complete characterisation of any biased complete graph in which only certain $4$-circuits are required to be balanced. We state a slightly simplified version here. For a biased graph $(G,\cB)$ and a vertex $v$ of $G$, we write $(G,\cB) - v$ for the biased subgraph $(H,\cB')$ with $H = G -v$.

\begin{theorem}
	Let $n \ge 5$ and $(G,\cB)$ be a biased graph with $G = K_{n}$. If, for all $v_1,v_2,v_3,v_4 \in V(G)$ for which either $v_1<v_2<v_4<v_3$ or $v_1<v_3<v_2<v_4$, the circuit with edge set $\{v_1v_2,v_2v_3,v_3v_4,v_4v_1\}$ is balanced, then $(G,\cB)-n = \ka{a}{n-1}$ for some $a > 0$. 
\end{theorem}  

The single vertex deletion turns out to be necessary; in fact,  Theorem~\ref{counterexamples} will show that if the statement is strengthened to assert that $(G,\cB)$ itself has the form $\ka{a}{n}$, then there are doubly exponentially many counterexamples.

\section{Ordered biased graphs}

Our biased graphs hereon are always complete. While we could phrase all our results in terms of graphs $K_n$ with vertex set $\{1, \dotsc, n\}$, when passing to subgraphs it is more convenient to state them in terms of biased graphs with a fixed linear ordering on the vertex set; this will also allow us to state a slightly stronger, `ordered' version of our main theorem. An \emph{ordered graph} is a pair $(G,\le)$ where $G$ is a graph and $\le$ is a total ordering on $V(G)$, and an \emph{ordered biased graph} is a triple $(G,\cB,\le)$, where $(G,\cB)$ is a biased graph and $(G, \le)$ is an ordered graph. Two ordered biased graphs $(G_1,\cB_1,\le_1)$ and $(G_2,\cB_2,\le_2)$ are \emph{isomorphic} if there is a graph isomorphism from $G_1$ to $G_2$ that preserves both the property of being balanced for a circuit, and the ordering of the vertices. We defined our special biased graphs $K^{\mathbf{u}}(n),K^{\mathbf{o}}(n)$ and $K^a(n)$ to each have a fixed vertex set $[n]$; from here on we view them as ordered biased graphs with the usual ordering on $[n]$.

We write either $[v_1, \dotsc, v_k]$, or $[v_i\colon i \in \bZ_k]$ for the circuit with edge set $\{v_kv_1, v_1v_2, \dotsc, v_{k-1}v_k\}$. Each $k$-circuit has $2k$ such representations. Given a circuit $C$ in an ordered graph $(G,\le)$, define the \emph{canonical ordering} of $C$ to be the unique sequence $v_1, \dotsc, v_k$ for which $C = [v_1, \dotsc, v_k]$ and $v_1 = \min(V(C))$ while $v_2 < v_k$; this distinguished ordering was considered but not named in the introduction. We say that two $k$-circuits $C,C'$ are \emph{similar} in $(G,\le)$ if their canonical orderings $v_1, \dotsc ,v_k$ and $v_1',\dotsc, v_k'$ are such that $v_i < v_j$ if and only if $v_i' < v_j'$. This is an equivalence relation on the set of all $k$-circuits of $G$, and for each $k$-element subset $X$ of $G$, each equivalence class contains precisely one circuit with vertex set $X$. 

We will be especially concerned with two special similarity classes of $4$-circuits. Call a $4$-circuit $C$ a \emph{1324-circuit} if its canonical ordering $v_1,v_2,v_3,v_4$ satisfies $v_1<v_3<v_2<v_4$ and define a \emph{1243-circuit} similarly. Note that if $V = [n]$ with $n \ge 4$, then $[1,3,2,4]$ and $[1,2,4,3]$ are natural examples of these two types of circuit. 

A circuit $[v_i \colon i \in \bZ_k]$ is \emph{oscillating} if for each $i \in \bZ_k$, the vertex $v_i$ satisfies either $v_i < \min(v_{i-1},v_{i+1})$ or $v_i > \max(v_{i-1},v_{i+1})$. This clearly does not depend on the choice of indexing, and is equivalent to the definition earlier that the edges of $C$ alternate between `positive' and `negative'. Oscillating $4$-circuits are precisely the $1243$-circuits.

For each circuit $C = [v_i\colon i \in \bZ_k]$, let $\delta(C) = |\delta_0(C)|$, where
\[\delta_0(C) = |\{i \in \bZ_k \colon v_i > v_{i-1}\}| - |\{i \in \bZ_k \colon v_i < v_{i-1}\}|. \]
Cyclically permuting the $v_i$ does not change either term in $\delta_0$, while reversing the order of the $v_i$ swaps the two terms, changing only the sign of $\delta_0$; thus, the parameter $\delta(C)$ is a well-defined function of $C$ itself, not on the given indexing of its vertices. Fixing the ordering of $V(C)$ to be the canonical one, we see that the value of $\delta$ is the absolute difference between the number of `positive' and `negative' edges in $C$ as defined in the introduction. Oscillating circuits clearly satisfy $\delta = 0$, since the two sets defining $\delta_0$ have equal size. Note that $\delta(C) \le |V(C)|-2$ and that $\delta(C)$ and $|V(C)|$ have the same parity.

Given these definitions, we can now prove the combinatorial characterisations of the special biased graphs $\kub{n},\kosc{n}$ and $\ka{a}{n}$ that were stated earlier. Recall that they all arise from group-labellings of the orientation $\oa{K_n}$ of $K_n$ in which edges are oriented `upwards'; the definitions of the relevant labellings are recalled in the proof below. 

\begin{lemma}\label{basicclasses}
	If $n \ge 1$ is an integer, then 
	\begin{enumerate}[(i)]
		\item\label{ubcase} The set of balanced circuits of $K^{\bu}(n)$ is empty,
		\item\label{osccase} The balanced circuits of $K^{\bo}(n)$ are precisely the oscillating circuits of $(K_n,\le)$, and
		\item\label{onecase} if $a \ge 0$ is an integer, then the balanced circuits of $K^a(n)$ are precisely the circuits $C$ of $(K_n,\le)$ for which $a|\delta(C)$. 
	\end{enumerate}
\end{lemma}
\begin{proof}
	Let $G = (V,E) \cong K_n$. Let $\Gamma$ be a cyclic group and $S \subseteq \Gamma$ be an $\binom{n}{2}$-element set whose subsets all have distinct products. 
	
	Let $\gamma = \gamma^{\mathbf{u}} \colon E \to \Gamma$ be a function assigning distinct elements of $S$ to each edge, so $\kub{n}$ arises from $(\oa{K_n},\gamma^{\bu})$. . Let $C = [v_1, \dotsc, v_k]$ be a cycle of $K_n$. Then $\pi(C)= \prod_{i \in \bZ_k} (\gamma^{\bu}(v_iv_{i+1}))^{-b_i}$ where each $b_i$ is $-1$ or $1$. So $\pi(C)$ is the identity if and only if $\prod_{i\colon b_i = 1}\gamma^{\bu}(v_iv_{i+1}) = \prod_{i \colon b_i = -1}\gamma^{\bu}(v_{i}v_{i+1})$. The two sides of this equality are products of distinct subsets of $S$, so this cannot occur; thus, all circuits of $K^{\bu}(n)$ are unbalanced, as required.
	
	Now let $\gamma^{\bo}\colon E \to \Gamma$ be a function assigning an element of $S$ to each edge, such that two edges receive the same label if and only if their larger end (head) is the same, so $\kosc{n}$ arises from $(\oa{K_n},\gamma^{\bo})$. For each $v \in [n]$, let $s_v \in S$ be the label assigned to the edges whose head is $v$. Let $C = [v_i \colon i \in \bZ_k]$. Since each edge is oriented from its minimum to its maximum, we have
	\begin{align*}
	\pi(C) &= \prod_{i \in \bZ_k \colon v_{i+1} > v_i}\gamma^{\bo}(v_iv_{i+1})\prod_{i \in \bZ_k \colon v_{i+1} < v_i}\gamma^{\bo}(v_iv_{i+1})^{-1} \\
	&= \left(\prod_{i \in \bZ_k \colon v_{i+1} > v_i}\gamma_{i+1}\right)\left(\prod_{i \in \bZ_k \colon v_{i+1} < v_i}\gamma_i\right)^{-1}. \\
	\end{align*}
	Note that since $v_1, \dotsc, v_k$ are distinct, each of the two terms is a product of distinct elements of $S$. Therefore, by the choice of $S$, the circuit $C$ is balanced if and only if $\{i \in \bZ_k\colon v_{i+1} > v_i\} = \{i \in \bZ_k\colon v_{i+1} < v_i\}$. This occurs precisely when each $v_i$ either satisfies $v_i > \min(v_{i-1},v_{i+1})$ or $v_i > \max(v_{i-1},v_{i+1})$: that is, when $C$ is oscillating. This gives (\ref{osccase}). 
	
	Finally, let $a \ge 0$ be an integer and let $\Gamma_a$ be a cyclic group with generator $g$ such that $\Gamma$ is infinite if $a = 0$, and $|\Gamma|$ has order $a$ otherwise. Let $\gamma^a\colon E \to \Gamma_a$ be the labelling assigning $g$ to every edge, so $\ka{a}{n}$ arises from $(\oa{K_n},\gamma^a)$. Then, for a cycle $C = [v_i\colon i \in \bZ_{k}]$, we have
	\begin{align*}
	\pi(C) &= \prod_{i \in \bZ_k \colon v_{i+1} > v_i}\gamma^{a}(v_iv_{i+1})\prod_{i \in \bZ_k \colon v_{i+1} < v_i}\gamma^{a}(v_iv_{i+1})^{-1} \\
	&= g^{|\{i \in \bZ_k\colon v_{i+1} > v_i\}| - |\{i \in \bZ_k\colon v_{i+1} < v_i|\}} = g^{\delta(C)}.
	\end{align*}
	Since $g^{\delta(C)}$ is the identity if and only if $a|\delta(C)$, the characterisation (\ref{onecase}) follows. 
\end{proof}

We also prove an easy folklore estimate for the number of (unordered) paths and circuits in a complete graph. This only serves to make the bounds in our final theorem slightly cleaner.  
\begin{lemma}\label{pathscycles}
$K_n$ has at most $2n!$ paths and at most $2(n-1)!$ circuits. 
\end{lemma}
\begin{proof}
	Let $p_n,c_n$ denote the number of paths and circuits respectively in $K_n$. Using $4-e > 1$, we have $p_n = n+ \sum_{i=2}^n \binom{n}{i}\tfrac{i!}{2} < \tfrac{1}{2} n + \tfrac{e}{2}n! < 2n!$.
	That $c_n \le 2(n-1)!$ is easy to verify for $n \le 4$; suppose that $n \ge 5$. Each circuit of $K_n$ that is not a circuit of $K_{n-1}$ corresponds to a nonempty path of $K_{n-1}$, together with vertex $n$ added, so $c_n \le p_{n-1} + c_{n-1}$ for each $n \ge 4$. Now, we inductively get $c_n \le c_{n-1} + p_{n-1} \le 2(n-2)! + \tfrac{e}{2}(n-1)! < 2(n-1)!$ (where we use $n -1 > \tfrac{4}{4-e}$), and the bound follows. 
\end{proof}

\section{Bipartite Graphs}

In this section and what follows, we use the following statement of Ramsey's theorem, as well as the very well-known bound $R_2(t,t) \le 2^{2t}$. 

\begin{theorem}[Ramsey's Theorem]\label{ramsey}
	There is a function $R_k(n_1,\dotsc,n_{\ell})$ so that, if $Z$ is a finite set with $R_k(n_1,\dotsc,n_{\ell})$ and each element of $\binom{Z}{k}$ is assigned a colour from $\{c_1,\dotsc,c_{\ell}\}$, then there is some $i \in [\ell]$ and a set $X \subseteq Z$ for which $|X| = k_i$ and each set in $\binom{X}{k}$ receives colour $c_i$. 
\end{theorem}

Call a set of circuits $\cC$ in a biased graph $(G,\cB)$ \emph{consistent} if either $\cC \subseteq \cB$ or $\cC \cap \cB = \varnothing$. Call a biased graph $(G,\cB)$ consistent if the set of all its circuits is consistent. For a set $X \subseteq V(G)$, write $(G,\cB)|X$ for the biased subgraph $(H,\cB')$ where $H$ is the subgraph of $G$ induced by $X$. We now restate and prove Theorem~\ref{mainbp}.

\begin{theorem}
	Let $t \ge 1$ be an integer. If $n \ge 2^{2^{2^{4t}}}$ and $(G,\cB)$ is a biased graph with $G \cong K_{n,n}$, then $(G,\cB)$ has a consistent $K_{t,t}$-subgraph.
\end{theorem}
\begin{proof}
	Let $n_0 = 4^t$; note that since $\binom{n_0}{2} < 2^{4t-1}$, we have $n^{2^{-\binom{n_0}{2}}} >2^{2^{t}} > t + 2(2t-1)!$. Let $(A,B)$ be a bipartition of $G$. Let $A_0$ be an $n_0$-element subset of $A$. Let $X_1,X_2, \dotsc, X_{\binom{n_0}{2}}$ be an enumeration of $\binom{A_0}{2}$. Let $k$ be maximal so that $0  \le k \le \binom{n_0}{2}$ and there is a set $B_k \subseteq B$ with $|B_k| \ge n^{2^{-k}}$ such that $(G,\cB)|(X_i \cup B_k)$ is consistent for each $1 \le i \le k$.
	
	Suppose that $k < \binom{n_0}{2}$. Let $H$ be a graph with vertex set $B_k$ in which vertices $u,v$ are adjacent if and only if the $4$-circuit $C_{uv}$ with vertex set $\{u,v\} \cup X_{k+1}$ is balanced. Since the $4$-circuits $C_{uv},C_{uw}$ and $C_{vw}$ form a $\Theta$-graph for all distinct $u,v,w$, adjacency in $H$ is an equivalence relation, so there is a set $B_{k+1} \subseteq B_k$ with $|B_{k+1}| \ge \sqrt{|B_k|} \ge n^{2^{-(k+1)}}$ that either contains no edges of $H$, or induces a complete subgraph of $H$. Now $(G,\cB)|(X_i \cup B_{k+1})$ is consistent for each $1 \le i \le k+1$, so the existence of $B_{k+1}$ contradicts the maximality of $k$. 
	
	We may therefore assume that $k = \binom{n_0}{2}$; let $B' = B_{k}$, so $|B'| \ge n^{2^{-k}} > t + 2(2t-1)!$ as noted earlier. Now $(G,\cB)|(X \cup B')$ is consistent for each $X \in \binom{A_0}{2}$; since $n_0 = 4^t \ge R_2(t,t)$, there is a $t$-element set $A' \subset A_0$ such that the set of $4$-circuits of $(G,\cB)|(A' \cup B')$ is consistent. 
	
	Let $B'' \subseteq B'$ be maximal so that every circuit in $G|(A' \cup B'')$ is unbalanced. We may assume that $|B''| < t$, as otherwise $G|(A' \cup B'')$ has a consistent $K_{t,t}$-subgraph; thus $|A' \cup B''| \le 2t-1$. By maximality, each $b \in B' - B''$ is in some balanced circuit $[b,x_1, \dotsc, x_{\ell}]$ where $x_1, \dotsc, x_{\ell}$ is a path. There are at most $2(2t-1)! < |B' -B''|$ such paths, so there exist $b,b' \in B''-B'$ and $x_1, \dotsc, x_{\ell} \in A' \cup B''$ for which $[b,x_1, \dotsc, x_{\ell}] \in \cB$ and $[b',x_1 \dotsc, x_{\ell}] \in \cB$. By the $\Theta$-property we have $[b,x_1,b',x_{\ell}] \in \cB$; since the set of $4$-circuits of $G|(A' \cup B')$ is consistent, it follows that every $4$-circuit of $G|(A' \cup B')$ is balanced. 
	
	We now argue that every circuit of $G|(A' \cup B')$ is balanced; since $|B'| \ge |A'| = t$, this implies the result. Indeed, a minimum-length unbalanced circuit $C = [c_1, \dotsc, c_{2k}]$ must have length greater than $4$, but then $[c_1,c_4,c_5, \dotsc, c_k]$ and $[c_1,c_2,c_3,c_4]$ are smaller (and thus balanced) circuits forming a theta with $C$, a contradiction. 	
\end{proof}

\section{Cliques}

We first prove a Ramsey-type result that will essentially find one of our three unavoidable examples in a large biased clique; however, recognising that the final outcome is (almost) $\ka{a}{t}$ for some $a$ will require a lemma from the next section. The proof of the following lemma uses the estimates from Lemma~\ref{pathscycles}.

\begin{lemma}\label{slog}
	Let $r,s,t \ge 1$ and $n \ge R_4(3r!,3s!,5,t)$. If $(G,\cB,\le)$ is an ordered biased graph with $G \cong K_n$, then either
	\begin{itemize}
		\item $G$ contains a $K_r$-subgraph in which no circuit is balanced,
		\item $G$ contains a $K_s$-subgraph in which the balanced circuits are precisely the oscillating circuits, or 
		\item $G$ contains a $K_t$-subgraph in which all $1243$-circuits and $1324$-circuits are balanced. 
	\end{itemize}
\end{lemma}
\begin{proof}
	Suppose for a contradiction that $G$ has none of the above subgraphs. 
	Let $(m_{\bal,\bal},m_{\bal,\ub},m_{\ub,\bal},m_{\ub,\ub}) = (3r!,3s!,5,t)$. For each $Q \in \binom{V}{4}$, let $c(Q) = (\alpha,\beta)$, where $\alpha = \bal$ if the unique $1423$-circuit on $X$ is balanced, and $\alpha = \ub$ otherwise, and $\beta \in \{\bal,\ub\}$ is defined similarly for the unique $1243$-circuit on $X$. By the definition of $n$, there is a subgraph $H$ of $G$ and some $(\alpha_0,\beta_0) \in \{\bal,\ub\}^2$ such that $c(Q) = (\alpha_0,\beta_0)$ for all $Q \in \binom{V(H)}{4}$, and $|V(H)| = m_{\alpha_0,\beta_0}$. Let $U = V(H)$. 
	
	\begin{claim}\label{mtc1}
		If $\beta_0 = \bal$ then $\alpha_0 = \bal$. 
	\end{claim}
	\begin{proof}
		Suppose not; so $(\alpha_0,\beta_0) = (\ub,\bal)$ and $|U| = 5$; let $v_1,\dotsc,v_5$ be the vertices of $U$ in increasing order. The $\Theta$-graph containing circuits $[v_1,v_2,v_4,v_3],[v_1,v_2,v_5,v_3]$ and $[v_2,v_4,v_3,v_5]$ contains two (balanced) $1243$-circuits and one (unbalanced) $1324$-circuit, a contradiction.  
	\end{proof}
	
	
	\begin{claim}\label{mtc2}
		$\alpha_0 = \bal$. 
	\end{claim} 
	\begin{proof}
		If not, then $(\alpha_0,\beta_0) = (\ub,\ub)$, so $|U| = 3r!$. Let $X \subseteq U$ be maximal such that $X$ contains no balanced circuits. If $|X| \ge r$ then $G$ has a $K_r$-subgraph with no balanced circuits, a contradiction; thus $|X| \le r-1$, so $U-X$ has a partition into at most $r$ intervals, and $|U - X| > 3r!- r > 2r(r-1)!$. By maximality, for each $v \in U - X$ there is a balanced circuit $C(v)$ contained in $X \cup \{v\}$. For each such $v$, the graph $P(v) = C(v)-v$  is a path contained in $X$, and there are at most $2(r-1)!$ such paths; by a majority argument there exist $v_1,v_2 \in U-X$ for which $P(v_1) = P(v_2)$, and $v_1$ and $v_2$ lie in the same interval of $U-X$. Now $C(v_1)$ and $C(v_2)$ form a $\Theta$-subgraph with a $4$-circuit $C = [u,v_1,w,v_2]$ where $u,w$ are the ends of $P(v_1)$, so $C$ is balanced. Since $u$ and $w$ do not lie in the interval between $v_1$ and $v_2$, $C$ is either a $1324$-circuit or a $1243$-circuit. Thus either $\alpha_0 = \bal$ or $\beta_0 = \bal$, and the claim follows from $\ref{mtc1}$. 
	\end{proof}

	The above claim gives that all $1324$-circuits in $H$ are balanced. 

	\begin{claim}
		All oscillating circuits in $H$ are balanced.  
	\end{claim}
	\begin{proof}
		Suppose otherwise, and let $C$ be an unbalanced oscillating circuit with $t$ as small as possible, with canonical ordering $v_1, \dotsc, v_t$. Note that $v_1 < v_3 < \min(v_2,v_4)$ and $t$ is even. Clearly $t > 4$, since all oscillating circuits of length at most $4$ are $1324$-circuits. The circuit $C' = [v_1,v_4,v_5 \dotsc, v_t]$ is also oscillating, and is balanced by minimality. Now $C$ and $C'$ are in a $\Theta$-graph with the (balanced) $1324$-circuit $[v_1,v_2,v_3,v_4]$, contradicting the choice of $C$ as unbalanced. 
	\end{proof}

	\begin{claim}
		$\beta_0 = \bal$. 
	\end{claim}
	\begin{proof}
		Suppose otherwise, so $(\alpha_0,\beta_0) = (\bal,\ub)$ (i.e. all $1243$-circuits of $H$ are unbalanced) and $|U| = 3s!$. Call a circuit \emph{bad} if it is balanced but not oscillating.  Let $n_1 = 2(s-1)!$, and let $U_0,U_1, \dotsc, U_{s}$ be sets of vertices of $H$ such that $|U_0| = s$, while for each $0 < i \le s$ we have $|U_i| = n_1$, and the vertices of $U_{i-1}$ precede the vertices of $U_i$ in the ordering; these sets exist because $s + sn_1 < s + 2s! < |U|$.

		
		Let $u_1, \dotsc, u_s$ be the vertices of $U_0$ listed in increasing order. Let $\cC$ denote the set of circuits of $H$ whose vertex set is contained in $U_0$ (noting that $|\cC| < n_1$), and let $\cW = U_1 \times U_2 \times \dotsc \times U_s$. For each $1 \le i \le s$ and $W = (w_1, \dotsc, w_s) \in \cW$, let $\varphi_W(u_s) = w_s$. For each $C \in \cC$, let $\varphi_W(C)$ denote the circuit obtained from $C$ by applying the map $\varphi_W$ to each vertex. 
		
		Note that the map $C \mapsto \varphi_W(C)$ is a bijection between $\cC$ and the set of circuits with vertex set contained in $W$, and that this map preserves the property of being oscillating.
		
		
		Let $W = (w_1, \dotsc, w_s) \in \cW$ be such that the number of bad circuits with vertex set contained in $W$ is as small as possible. If this number is zero, then $W$ is the vertex set of a $K_s$-subgraph in which precisely the oscillating circuits are balanced, a contradiction. Thus, there is some $C_0 \in \cC$ for which $\varphi_W(C_0)$ is bad, so, since $C_0$ is not oscillating, there are consecutive vertices $x,y,z$ in $C_0$ with $x <y<z$. Let $A$ be the set $U_i$ containing $\varphi_W(y)$, and for each $a \in A$, let $W_a$ be obtained from $W$ by replacing $\varphi_W(y)$ with $a$. Note that if $C \in \cC$, then either
		\begin{enumerate}
			\item\label{mte1} $y \notin V(C)$ and $\varphi_{W_a}(C) = \varphi_{W_{a'}}(C)$ for all $a,a' \in A$, 
			\item\label{mte2} $y$ is either the minimum or maximum of its two neighbours in $C$ and, for all $a,a' \in A$, the circuits $\varphi_{W_a}(C)$ and $\varphi_{W_{a'}}(C)$ form a $\Theta$-graph with a $1324$-circuit (which is balanced by (\ref{mtc2})), or 
			\item\label{mte3} $y$ lies between its two neighbours in $C$ and, for all $a,a' \in A$, the circuits $\varphi_{W_a}(C)$ and $\varphi_{W_a'}(C)$ form a $\Theta$-graph with an $1243$-circuit (which is unbalanced by assumption). 
		\end{enumerate} 
		
		For each $C \in \cC$, consider the set $A_C$ of elements $a \in A$ for which $\varphi_{W_a}(C)$ is balanced and $\varphi_W(C)$ is unbalanced. Note that if (\ref{mte1}) holds then $A_C = \varnothing$. If (\ref{mte2}) holds, then for each $a \in A_C$, the $\Theta$-graph containing some $1324$-circuit and the circuits $\varphi_W(C)$ and $\varphi_{W_a}(C)$ contains exactly two balanced circuits, so $A_C = \varnothing$. If (\ref{mte3}) holds and $a,a' \in A_C$ are distinct, then the $\Theta$-graph containing $\varphi_{W_a}(C)$, $\varphi_{W_{a'}}(C)$, and some $1243$-circuit contains exactly two balanced circuits, so $|A_C| \le 1$. 
		
		Since $|\cC| < n_1 = |A|$, there is some $b \in A$ that is outside every $A_C$. Note that $\varphi_W(C_0)$ is balanced and that, since $y$ lies (in the ordering) between its two neighbours in $C_0$, the circuit $C_0$ satisfies (\ref{mte3}) and, since $1243$-circuits are unbalanced, so is the circuit $\varphi_{W_b}(C_0)$. So $\varphi_{W_b}(C_0)$ is not bad. Moreover, by choice of $b$, there is no $C \in \cC$ for which $\varphi_W(C)$ is not bad and $\varphi_{W_b}(C)$ is bad. Therefore $W_b$ contains strictly fewer bad circuits than $W$, contradicting the choice of $W$.   
	\end{proof}
	By these claims, we have $(\alpha_0,\beta_0) = (\bal,\bal)$ and $|U| = t$; thus, $H$ satisfies the third outcome, a contradiction.
\end{proof}

\section{Balanced $4$-circuits}

Let $(G,\le)$ be an ordered graph where $G$ is complete with vertex set $V$. Define $\Omega(G)$ to be the graph whose vertices are the \emph{nonspanning} circuits of $G$, in which two circuits $C_1,C_2$ are adjacent if and only if there is a $\Theta$-subgraph of $G$ containing $C_1,C_2$ and some $1243$-circuit or $1324$-circuit. Our first lemma describes the components of $\Omega$.

\begin{lemma}\label{omegacomponents}
	Let $n \ge 5$ and let $(G,\le)$ be an ordered graph where $G \cong K_n$. Then two vertices $F_1,F_2$ of $\Omega(G)$ are in the same component of $\Omega$ if and only if $\delta(F_1) = \delta(F_2)$. 
\end{lemma}
\begin{proof}
	We may assume that $G = K_n$ and $\le$ is the usual ordering. Let $\Omega = \Omega(G)$. Say that two vertices of $\Omega$ are \emph{connected} if they are in the same component of $\Omega$. 
	
	\begin{claim}\label{similar}
		Each pair of similar circuits of $G$ is connected in $\Omega$.  
	\end{claim}
	\begin{proof}[Subproof:]
		It suffices to show for each $3 \le k \le n$ that each $k$-circuit of $G$ is connected in $\Omega$ to the unique circuit with vertex set $[k]$ that is similar to $C$. 
		Let $C_0$ be a $k$-circuit such that $C,C_0$ are connected in $\Omega$, while the sum $\Sigma(C_0)$ of the vertices in $C_0$ is as small as possible. If $V(C_0) = [k]$ then the required statement holds. Otherwise there is some $\ell \in V(C_0)$ for which $\ell-1 \notin V(C_0)$. Let $C_0'$ be obtained by replacing the vertex $\ell$ with the vertex $\ell-1$ in $C'$. Now $C_0'$ and $C_0$ form a $\Theta$-graph with a $4$-circuit $C'' = [x,\ell-1,y,\ell]$; since no vertex lies between $\ell-1$ and $\ell$ in the ordering, we see that $C_4$ is a $1243$- or $1324$-circuit. Thus $X_0'$ and $C_0'$ are adjacent in $\Omega$; since $\Sigma C_0' = \Sigma C_0 - 1$, this contradicts the minimality in the choice of $C_0$. 
	\end{proof}
	
	Call a circuit $C$ of $G$ \emph{basic} if either $C = [1,2,4,3]$, or if $C = [1,\dotsc,k]$ for some $k \ge 3$. Note that the first type has $\delta = 0$, while other basic $k$-circuits have $\delta = k-2$. 
	
	\begin{claim}\label{basiccircuit}
		Each nonspanning circuit of $G$ is connected in $\Omega$ to a basic circuit.
	\end{claim}
	\begin{proof}[Subproof:]
		For each circuit $C$ of $G$ with canonical ordering $v_1,\dotsc,v_k$, call a pair $(i,j)$ \emph{decreasing} in $C$ if $1 \le i < j \le k$ and $v_i > v_j$; since the ordering is canonical, any decreasing pair $(i,j)$ satisfies $i > 1$.  Similar circuits have the same set of decreasing pairs.  
		
		Fix a circuit $C$ of $G$. Let $C'$ be a circuit connected to $C$ in $\Omega$ for which $|V(C')|$ is as small as possible, and, subject to this, $C'$ has as few decreasing pairs as possible. Let $C_0$ be the unique circuit with vertex set $[k]$ that is similar to $C'$; we show that $C_0$ is basic. Let $v_1, v_2, \dotsc, v_k$ be the canonical ordering of $C_0$, so $v_1 = 1$ and $v_2 < v_k$. 
		
		Note that $C$ is connected to $C_0$ in $\Omega$, and that the decreasing pairs for $C_0$ are exactly those for $C'$. If $C_0$ has no decreasing pair then $C_0 = [1,\dotsc,k]$ is basic, giving the result. If $k \le 4$ then $C_0$ is either one of the basic circuits $[1,2,3],[1,2,3,4],[1,2,4,3]$, or $C_0 = [1,3,2,4]$. In the first three cases the result holds, and in the last, the $\Theta$-graph containing circuits $C_0,[3,5,2,4]$ and $[1,3,5,4]$ certifies that $C_0$ and $[1,3,5,4]$ are adjacent in $\Omega$, since $[1,3,5,4]$ is similar to $[1,2,4,3]$, the result holds by (\ref{similar}).
		
		Suppose, therefore, that $k > 4$ and that $C_0$ has a decreasing pair. Let $(a,b)$ be a decreasing pair for $C_0$ for which $a$ is as small as possible, and subject to this, $b$ is as large as possible. Note that $a > 1$. Suppose first that $b = a+1$; since $(a-1,a+1)$ is not decreasing we have $v_{a-1} < v_{a+1} < v_a$. 
		
		If $a = 2$ then, since $(2,4)$ is not decreasing we have $v_1 < v_{3} < v_2 < v_4$. Let $C_1 = [v_1,v_4,v_5\dotsc,v_k]$; note that $|V(C_1)| = k-2 \ge 3$. Now $C_0$ and $C_1$ form a $\Theta$-graph with the $1324$-circuit $C_2 = [v_1,v_2,v_3,v_4]$. Therefore $C_1$ is adjacent to $C_0$ in $\Omega$ and is thus connected to $C$; since $|V(C_1)| < |V(C')|$, this contradicts the minimality in the choice of $C'$.  
		
		If $a > 2$ then since $(a-2,a-1)$ is not decreasing we have have $v_{a-2} < v_{a-1} < v_{a+1} < v_a$. Let $C_1 = [v_1,\dotsc,v_{a-2},v_{a+1}, \dotsc, v_k]$ (noting that $|V(C_1)| = k-2 \ge 3$). The circuits $C_0$ and $C_1$ form a $\Theta$-graph with the $1243$-circuit $C_2 = [v_{a-2},v_{a-1},v_a,v_{a+1}]$; we obtain a contradiction as in the previous case.  
		
		We may thus assume that $b > a+1$. We will now derive a contradiction by finding a $k$-circuit $C_0''$, connected to $C$ in $\Omega$, with fewer decreasing pairs than $C$. 
		
		Let $C_0' = [v_1',v_2',\dotsc,v_k']$ where $v_i' = v_i$ if $v_i < v_b$, and $v_i' = v_b+1$ if $v_i \ge v_b$. So $C_0'$ has vertex set $[k+1]-\{v_b\}$ and is similar to $C_0$ (and the ordering $v_1', \dotsc, v_k'$ is canonical). Let $C_0'' = [v_1'',\dotsc, v_k'']$, where $v_i'' = v_i'$ if $i \ne a$, and $v_a'' = v_b$. Again, since $v_2'' \le v_2' < v_k' = v_k''$, the ordering $v_1'', \dotsc, v_k''$ is canonical.   The circuits $C_0''$ and $C_0'$ form a $\Theta$-graph with the $4$-circuit $C_1 = [v_{a-1}',v_b,v_{a+1}',v_{a}']$. Since $v_b' = v_b+1 \not\in V(C_1)$, the circuit $C_1$ is similar to $C_1' = [v_{a-1}',v_b',v_{a+1}',v_a']$; the fact that $v_b' < v_a'$ and neither $(a-1,b)$ nor $(a,a+1)$ is decreasing in $C_0'$ gives $v_{a-1}' < v_b' < v_{a}' < v_{a+1}'$ and so $C_1'$ and $C_1$ are $1243$-circuits. This implies that $C_0'$ and $C_0''$ are adjacent in $\Omega$. 
		
		Since $v_a'' = v_b < v_b+1 = v_b''$, the pair $(a,b)$ is not decreasing in $C_0''$ but is decreasing in $C_0$ and therefore in $C'$. We now argue that every decreasing pair $(i,j)$ in $C_0''$ is decreasing in $C_0'$; this will imply that $C_0''$ has strictly fewer decreasing pairs than $C_0$ and therefore fewer such pairs than $C'$, contradicting the minimality in the choice of $C'$. 
		
		Let $(i,j)$ be a decreasing pair in $C_0''$. If $j = a$ then $i < a$ and, using the fact that $(a-1,b)$ is not decreasing in $C_0'$, we have $v_i'' = v_i' \le v_{a-1}' < v_b' = v_b+1 = v_a''+1$ and so $v_i'' \le v_a'' = v_j''$, contradicting the definition of $(i,j)$. If $j \ne a$ then, using the fact that $v_{\ell}' \ge v_{\ell}''$ for all $\ell$, we have $v_i' \ge v_i'' > v_j'' = v_j'$ so $(i,j)$ is decreasing in $C_0'$. This fact establishes the aforementioned contradiction. 
	\end{proof}

	\begin{claim}\label{preservedelta}
		If $C_1,C_2$ are connected in $\Omega$, then $\delta(C_1) = \delta(C_2)$. 
	\end{claim}
	\begin{proof}[Subproof:]
		It suffices to show this for $C_1,C_2$ adjacent in $\Omega$. Let $H$ be a $\Theta$-graph of $\Omega$ containing $C_1,C_2$ and a $1324$-circuit or $1243$-circuit $C_3$; note that $\delta(C_3) = 0$. Let $x,y$ be the vertices of $H$ for which $E(H)$ is the disjoint union of three $xy$-paths $P_1,P_2,P_3$, where each $C_i$ has edge set disjoint from that of $P_i$. For each $xy$-path $P$ with vertices $x = x_0,x_1,\dotsc,x_{\ell} = y$ listed in path order, let 
		\[\delta(P) = |\{i \in [\ell]\colon x_i > x_{i-1}\}| - |\{i \in [\ell]\colon x_i < x_{i-1}\}|.\]
		By construction, we have  \[(\delta(C_1),\delta(C_2),\delta(C_3)) = (|\delta(P_3)-\delta(P_2)|,|\delta(P_1)-\delta(P_3)|,|\delta(P_2)-\delta(P_1)|).\] Since $\delta(C_3) = 0$, this gives $\delta(P_1) = \delta(P_2)$ and so $\delta(C_1) = \delta(C_2)$. 
	\end{proof}
	By (\ref{preservedelta}), vertices of $\Omega$ in the same component have the same value of $\delta$. By (\ref{basiccircuit}), any two vertices $C_1,C_2$ of $\Omega$ with $\delta(C_1) = \delta(C_2) = d$ are both connected to the unique basic circuit $C_0$ with $\delta(C_0) = d$, so are connected in $\Omega$; the lemma follows. 
\end{proof}



\begin{lemma}\label{constantlabelling}
	Let $n \ge 5$ and let $(G,\cB,\le)$ be an ordered biased graph with $G \cong K_n$. If all $1324$-circuits and $1243$-circuits of $G$ are balanced, then there exists $a > 0$ such that $(G,\cB)-v \cong \ka{a}{n-1}$ for each $v \in V$. 
\end{lemma}
\begin{proof}
	We may assume that $(G,\le)$ is $K_n$ with the usual ordering. Let $\Omega = \Omega(G)$. Note that a nonspanning circuit $C$ of $G$ must have $\delta(C) \le |V(C)|-2 \le n-3$. For each $0 \le d \le n-3$, let $\cU_d$ be the set of nonspanning circuits $C$ of $G$ with $\delta(C) = d$. Note that each $\cU_d$ is nonempty; by Lemma~\ref{omegacomponents}, the $\cU_d$ are the components of $\Omega$. 
	
	By the definition of $\Omega$, hypothesis, and the $\Theta$-property, adjacent vertices of $\Omega$ are either both balanced or both unbalanced; it follows that for each $d \ge 0$, we either have $\cU_d \subseteq \cB$ or $\cU_d \cap \cB = \varnothing$. Let $\cD$ be the set of all $0 \le d \le n-3$ for which $\cU_d \subseteq \cB$.  Since $\cU_0 \cap \cB$ contains the circuit $[1,2,4,3]$, we have $0 \in \cD$. 

	Let $d_1,d_2 \in \cD$ with $0 < d_1 < d_2$; we show that $d_2-d_1 \in \cD$. Let $H$ be the $\Theta$-graph containing $C_2 = [1,2,\dotsc, d_2+2]$ together with a single edge from $1$ to $d_1 + 2$. We have $\delta(C_2) = d_2$, and the other two circuits $C_0,C_1$ of $H$ satisfy $\delta(C_0) = d_2-d_1$ and $\delta(C_1) = d_1$. Since $d_1,d_2 \in \cD$ we have $\{C_1,C_2\} \subseteq \cB$ and so $C_0 \subseteq \cB \cap \cU_{d_2-d_1}$ by the $\Theta$-property. This implies that $\cU_{d_2-d_1} \subseteq \cB$ and so $d_2-d_1 \subseteq \cD$. Thus $\cD$ contains zero and is closed under absolute differences. It follows that there exists $a > 0$ for which $\cD = \{d \in \{0,\dotsc,n-3\} \colon a|d\}$, and thus a nonspanning circuit $C$ of $G$ is balanced if and only if $a|\delta(C)$. Thus $(G,\cB)-v \cong K^a(n-1)$ for each vertex $v$ of $G$.
\end{proof}

Combining this Lemma with Theorem~\ref{slog} and Lemma~\ref{basicclasses}, we get the following restatement of our main theorem with explicit bounds in terms of Ramsey numbers and our stronger, `ordered' isomorphism. 

\begin{theorem}\label{main}
	Let $r,s \ge 1$, $t \ge 4$ and $n \ge R_4(3r!,3s!,5,t+1)$. If $(G,\cB,\le)$ is an ordered biased graph with $G \cong K_n$, then $(G,\cB,\le)$ has an ordered biased subgraph isomorphic to  $\kub{r}$, $\kosc{s}$, or $\ka{a}{t}$ for some $a > 0$. 
\end{theorem}

Finally, we give a construction implying that the vertex deletion in Lemma~\ref{constantlabelling} is necessary, in fact showing that there are doubly exponentially many counterexamples if the vertex deletion is not insisted upon. Since Lemma~\ref{constantlabelling} controls the behaviour of every nonspanning cycle, all the wildness occurs within the Hamilton cycles. 


\begin{theorem}\label{counterexamples}
	For all $n \ge 10$, there are at least $2^{2^{n-9}}$ ordered biased graphs $(G,\cB,\le)$ where $G = K_n$ and every $1243$-circuit and $1324$-circuit of $G$ is balanced, but $(G,\cB)$ does not arise from any group-labelling. 
\end{theorem}
\begin{proof}
	
	Let $\cP$ be the collection of ordered partitions $(I,J)$ of $[n]$ for which $\{1,2,n-2,n-1\} \subseteq I$ and $\{3,4,5,n\} \subseteq J$, so $|\cP| = 2^{n-8}$. Let $G = (V,E) = K_n$ and let $(G,\cB,\le) = K^{a}(n)$ where $a = n-4$.  For each $(I,J) \in \cP$, let $C_{I,J}$ be the $n$-circuit with canonical ordering $u_1,\dotsc,u_\ell,v_1, \dotsc, v_k$, where the $u_i$ are the elements of $I$ in increasing order and the $v_j$ are the elements of $J$ in increasing order. Note that $\delta(C_{I,J}) = n-4$ and therefore that $C_{I,J} \in \cB'$ for each $(I,J)$, and that the map $(I,J) \mapsto C_{I,J}$ is injective. For each $\cQ \subseteq \cP$, let $\cB_{\cQ} = \cB - \{C_{I,J}\colon (I,J) \in \cQ\}$.

	\begin{claim}
		$(G,\cB_{\cQ})$ is a biased graph for all $\cQ \subseteq \cP$.
	\end{claim}
	\begin{proof}[Subproof:]
		If not, then there is a $\Theta$-subgraph $H$ containing exactly two balanced circuits. No such subgraph exists in $(G,\cB)$, so it must be the case that $H$ contains an unbalanced $n$-circuit $C_0 = C_{I,J}$ for some $(I,J) \in \cQ$ and two other balanced circuits $C_1,C_2$; since no $\Theta$-subgraph of $K_n$ contains two $n$-circuits, we must have $C_1,C_2 \in \cB$. Let $P_0,P_1,P_2$ be the three $xy$-paths so that $E(C_i) = E(H) - E(P_i)$ for each $i \in \{0,1,2\}$. Since $C_0$ is an $n$-circuit, the path $P_0$ is a chord of $C_0$, so $|E(P_0)| = 1$ and $|E(P_1)| + |E(P_2)| = n$. As in the proof of (\ref{preservedelta}), let 
	\[\delta(P) = |\{i \in [\ell]\colon x_i > x_{i-1}\}| - |\{i \in [\ell]\colon x_i < x_{i-1}\}|\]
	for each path $P$ with vertices occuring in order $x_i\colon i \in [\ell]$; so
	\[(\delta(C_0),\delta(C_1),\delta(C_2)) = (|\delta(P_2)-\delta(P_1)|,|\delta(P_1)-\delta(P_0)|,|\delta(P_0)-\delta(P_2)|).\]
	We have $\delta(C_0) = n-4$ and, since $C_1,C_2 \in \cB$, both $\delta(C_1)$ and $\delta(C_2)$ are multiples of $n-4$. The above also gives $n-4 = \delta(C_0) = \pm \delta(C_1) \pm \delta(C_2)$. Moreover, since $\delta(C_i) \le |V(C_i)|-2 \le n-3$ for each $i \in \{1,2\}$, we must have $\{\delta(C_1),\delta(C_2)\} = \{0,n-4\}$; we may assume that $\delta(C_1) = 0$ (so $\delta(P_2) = \delta(P_0)$) and $\delta(C_2) = n-4$. Thus $n-4 = \delta(C_2) \le |V(C_2)|-2$ and so $|V(C_2)| \ge n-2$; since $\delta(C_2) \equiv |V(C_2)| \pmod 2$ and $|V(C_2)| < n$ we must have $|V(C_2)| = n-2$. Thus $|E(P_2)| = |E(H)| - |E(C_2)| = (n+1)-(n-2) = 3$ and so $C_1$ is a $4$-cycle with $\delta(C_1) = 0$. 
	
	However, the construction of $\cP$ and $C_0 = C_{I,J}$ gives that $C_0 = [u_1,u_2,\dotsc,u_{\ell},v_1,v_2,\dotsc,v_k]$ where the $u_i$ occur in increasing order with $(u_1,u_2,u_{\ell-1},u_{\ell}) = (1,2,n-2,n-1)$ and the $v_i$ occur in increasing order with $(v_1,v_2,v_3,v_k) = (3,4,5,n)$. The $4$-cycle $C_1$ satisfies $C_1 = [x_1,x_2,x_3,x_4]$ where $x_1,x_2,x_3,x_4$ are consecutive vertices in $C$. Any such quadruple of consecutive vertices is either a subsequence of $(u_1, \dotsc, u_\ell)$ or $(v_1, \dotsc, v_k)$, or is one of the quadruples $(n-2,n-1,3,4),(n-1,3,4,5),(v_{k-1},n,1,2)$ or $(n,1,2,u_3)$.  All these possibilities imply that $\delta(C_1) = 2$, a contradiction. 
	\end{proof}

	We now prove that if $\varnothing \ne \cQ \ne \cP$, then $(G,\cB_{\cQ})$ does not arise from a group-labelling. Since there are $2^{|\cP|} - 2 > 2^{2^{n-9}}$ possible $\cQ$, this will give the lemma. Suppose it does; let $\Gamma$ be a group with identity $\id$, and let $\gamma \colon E \to \Gamma$ be a labelling of the edges of some orientation $\oa{G}$ of $G$ whose balanced circuits are precisely those in $\cB_{\cQ}$. It is well-known and routine to show that applying either of the following two operations to a $\Gamma$-labelling of an orientation of a graph do not change the associated biased graph:
	\begin{enumerate}[(a)]
		\item\label{reorient} reversing the orientation of an edge, and replacing the label of that edge with its inverse in $\Gamma$;
		\item\label{scale} for a vertex $v$ and a constant $\alpha \in \Gamma$, right-multiplying the label of every edge oriented towards $v$ by $\alpha$, then left-multiplying the label of every edge oriented away from $v$ by $\alpha^{-1}$. 
	\end{enumerate}
	By possibly applying (\ref{reorient}) repeatedly, we may assume that every edge $ij$ with $i < j$ is oriented towards $j$. By then applying (\ref{scale}) with $\alpha = \gamma(1v)^{-1}$ to each vertex $v  \in \{2, \dotsc, n\}$, we may assume that $\gamma(1i) = \id$ for each $i > 1$. 
	Let $1 < p<q<r \le n$. Then the $1243$-circuit $C = [1,p,r,q]$ is balanced and so $\id = \gamma(1p)\gamma(pr)\gamma(qr)^{-1}\gamma(1q)^{-1} = \gamma(pr)\gamma(qr)^{-1}$, giving $\gamma(pr) = \gamma(qr)$. Similarly, the $1324$-circuit $C' = [1,q,p,r]$ is balanced, so $\id = \gamma(1q)\gamma(pq)^{-1}\gamma(pr)\gamma(r1)^{-1}$, giving $\gamma(pq) = \gamma(pr) = \gamma(qr)$. It follows that every pair of incident edges of $G - \{1\}$ are assigned the same label by $\Gamma$; connectedness gives that all edges of $G - 1$ are assigned the same label $g \in \Gamma$. Let $\gamma_0$ by the labelling obtained from $\gamma$ by applying (\ref{scale}) with $v = 1$ and $\alpha = g$; now $\gamma_0$ assigns a label of $g$ to every edge of $G$, and $\cB_{\cQ}$ is the set of balanced cycles of $(\gamma',\oa{G})$ by construction. But $\gamma_0$ is also a labelling of edges of $G$ with elements of the cyclic subgroup $\Gamma_0$ of $\Gamma$ generated by $g$. By the definition of $K^a(n)$ we therefore have $(G,\cB_{\cQ},\le) = K^a(n)$, where $a = 0$ if $\Gamma_0$ is infinite, and $a = |\Gamma_0|$ otherwise. This implies that a circuit $C$ of $G$ is in $\cB_{\cQ}$ if and only if $a|\delta(C)$. However, if $(I,J) \in \cQ$ and $(I',J') \in \cP-\cQ$, we have $C_{I,J} \notin \cB_Q$ and $C_{I',J'} \in \cB_{\cQ}$ while $\delta(C_{I,J}) = \delta(C_{I',J'}) = n-4$, a contradiction. 
\end{proof}

\section{References}
\newcounter{refs}

\begin{list}{[\arabic{refs}]}
{\usecounter{refs}\setlength{\leftmargin}{10mm}\setlength{\itemsep}{0mm}}

\item\label{dfp}
M. Devos, D. Funk and I. Pivotto, 
When does a biased graph come from a group labelling?,
Adv. Appl. Math. 61 (2014), 1--18. 

\item\label{zaslavsky}
T. Zaslavsky, 
Biased graphs. I. Bias, balance, and gains,
J. Combin. Theory Ser. B 47 (1989), 32--52.

\end{list}
\end{document}